\documentclass{article}
\usepackage{ijcai22}

\usepackage{times}

\usepackage[hidelinks]{hyperref}
\usepackage[utf8]{inputenc}
\usepackage[small]{caption}
\usepackage{graphicx}
\usepackage{comment}
\usepackage{amsmath,amsthm,amssymb,amsfonts,url,hyperref}
\usepackage{xcolor}
\usepackage{soul}
\usepackage{booktabs}
\usepackage{algorithm}
\usepackage[noend]{algpseudocode}

\newcommand{\mB}{\ensuremath{\mathcal{B}}}

\newcommand{\mH}{\ensuremath{\mathcal{H}}}

\newcommand{\mN}{\ensuremath{\mathcal{N}}}

\newcommand{\mP}{\ensuremath{\mathcal{P}}}

\newcommand{\hRb}{\ensuremath{\hat{R}_b}}
\newcommand{\bRb}{\ensuremath{\bar{R}_b}}

\newcommand{\Nmax}{\ensuremath{N_{\max}}}
\newcommand{\zph}{\ensuremath{z^{p,h}}}
\newcommand{\yphb}{\ensuremath{y^{p,h}_b}}
\newcommand{\yphbd}{\ensuremath{y^{p,h}_{b''}}}
\newcommand{\yphbs}{\ensuremath{y^{p,h}_{b^*}}}
\newcommand{\hzph}{\ensuremath{\hat{z}^{p,h}}}
\newcommand{\hyphb}{\ensuremath{\hat{y}^{p,h}_b}}
\newcommand{\hzpM}{\ensuremath{\hat{z}^{p,M}}}
\newcommand{\hypMb}{\ensuremath{\hat{y}^{p,M}_b}}
\newcommand{\zdp}{\ensuremath{\hat{z}^{p}}}
\newcommand{\Tdp}{\ensuremath{\hat{T}^p}}
\newcommand{\ydpb}{\ensuremath{\hat{y}^p_b}}
\newcommand{\xdpb}{\ensuremath{\hat{x}^p_b}}
\newcommand{\mNgz}{\ensuremath{\mN^{>0}}}

\newtheorem{model}{Formulation}
\newtheorem{proposition}{Proposition}

\title{A Rapid Algorithm for Beam Illumination Patterns and Hopping Time Plan}

\author{
Angus L. Gaudry$^1$\footnote{Contact Author}\and
Vicky Mak-Hau$^2$
\\
\affiliations
$^1$School of Information Technology, Deakin University, Waurn Ponds, Geelong, Vic 3215, Australia\\
\emails
\{agaudry, vicky\}@deakin.edu.au}

\begin{document}

\maketitle

\begin{abstract}

Beam hopping (BH) is a satellite communications technique in which sets of beams are sequentially illuminated over a defined time interval. Geographically varying the duty cycle of satellite transmission allows for reduced resource wastage as satellite capacity is matched to non-uniform user demands.  
Total feasible active beam combinations is given by $2^n$ for $n$ beams. With in service satellite systems operating in excess of 100 beams, complete enumeration of optimum illumination patterns is not tractable. Developing efficient optimization methods which minimize resource wastage is essential to realising the benefits of modern BH systems.  
We present a computationally efficient pattern generation method which uses the binary logarithm to decompose beam demands into common powers of two. Patterns are generated by beams sharing common powers with illumination times being a function of the magnitude of the power.
This method is shown to produce feasible patterns within 0.047 and 0.31 seconds for systems using 49 and 132 beams respectively and within 19.109 seconds for a 1085 beam system. 
When averaged across all testing configurations, this method reduced capacity error by 94\% compared to a conventional even data distribution. 
To facilitate algorithm comparison, two integer linear programming formulations are developed. Even though they can only solve problems of modest sizes to proven optimality, they provide valuable insights in the optimality gap for our proposed heuristic approach. 
\end{abstract}

\section{Introduction}

Contemporary satellite communication (SATCOM) problems are defined by non-uniform geographically dispersed user demands. Therefore capability trends in addressing these problems are characterised by flexibility and efficiency in order to minimize resource cost while adapting to changing user requirements. Beam hopping (BH) is one such capability which uses time domain switching to adapt supplied data rate to the requested demand per beam. This is achieved by cycling sets of illuminated beams over a short time period to create a pseudo-continuous data supply. The illumination duration of each set of beams is adjusted to meet the user demands across the complete cycle - reducing over/under supply of resources to beams.  

A satellite spot-beam is defined as a discrete area of radiated power, projected onto a surface by a satellite's antenna. A user located within the defined beam area is able to receive SATCOM data relative to the radiated power, with a beam having a non-zero radiated power being defined as ``illuminated''. The set of concurrently illuminated beams for a given time instance is defined as the beam illumination pattern (BIP) and the complete set of BIPs to be repeated over a defined time interval is the beam hopping time plan (BHTP). See Figure \ref{Dia:solution} for an example. Current in service High Throughput Satellites (HTS) operate in the order of 100 beams, however future systems such as the Viasat-3 satellite are expected to use over 1000 beams. While increased beam numbers allow for greater data throughput and flexibility, it significantly increases resource management complexity as the solution space for optimization problems associated with the satellite increases.

\begin{figure*}
  \includegraphics[width=\textwidth]{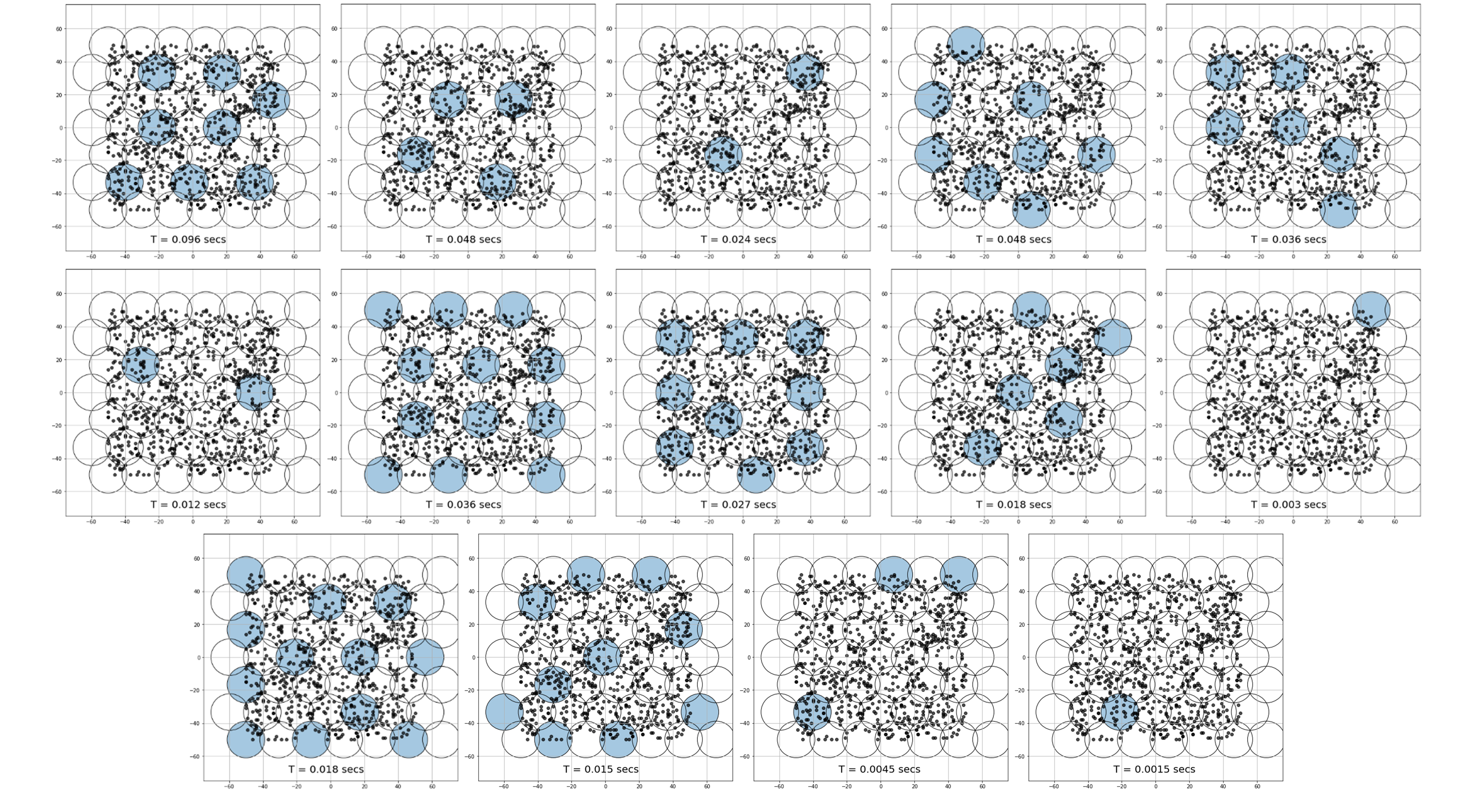}
  \caption{Example of a beam hopping cycle with 14 beam illumination patterns and their associated illumination times. Black dots are users and the circles are the beam footprints.}
  \label{Dia:solution}
\end{figure*}

\section{Defining an underlying mathematical problem}\label{sec2}

A BHTP comprises a set of BIPs each with an associated illumination duration (a.k.a. Beam Dwell Time (BDT)). Inter-beam interference must be considered when developing BIPs, In our implementation, we assume that the full bandwidth is supplied to each active beam, with a radiation pattern where interference is negligible beyond adjacent beams. If a dual polarisation frequency reuse system is considered, such as those seen in typical HTS applications, neighbouring beams could be illuminated while avoiding interference. However, this would reduce the flexibility of the beam hopping system to assign the sum of available frequencies to a single beam. 

The BH Optimization Problem (BHOP) is to find a BHTP such that the overall supplied data rate for each beam is as close to the requested data rate as possible subject to a number of equipment constraints, including the cardinality constraint (at most $\Nmax$ beams simultaneously illuminated in a time instance) and neighbouring beams interference constraints. The implementation of a BHTP is constrained by the BIP minimum illumination granularity. This will be determined by the required terminal synchronization time of the utilised waveform. The DVB-S2X standards, which are used extensively in BH literature, provide a verified waveform for use in a BH system. In this standard, the illumination time of a BIP is determined by the Super-Frame (SF) structure, which can either be fixed duration blocks or variable duration depending on the waveform chosen. 
Let: 
\begin{itemize}
    \item $N_B$ be the number of satellite spot-beams; 
    \item $\mN = \{1,\ldots,N_B\}$ be the index set of beams; 
    \item $T_H$ duration of a BH cycle; 
    \item $d$ the duration of a SF (i.e., the duration of a time slot if fixed SF length is considered); 
    \item $W$ number of (fixed length) SF in a BH cycle (hence, $T_H = d \times W$); 
    \item $\hRb$ the requested demand in Beam $b$; and 
    \item $\bRb$ the supplied data rate for Beam $b$. 
\end{itemize} 
A trivial solution can be obtained by one beam per BIP with an illumination time of 
$\hRb T_H /(\sum_{b \in \mN} \hRb)$ 
in a BH cycle with flexible-length SF, and 
$ \hRb W/(\sum_{b \in \mN} \hRb)$ 
rounded to the nearest integer for a fixed-length SF. 
However, at the transition of each BIP, time is required to account for beam switching and terminal synchronization, this is captured in the pre/post-amble frames of the SF structure and is defined as the switching time. For a fixed time interval, the effective data transmission time decreases with an increase of cumulative switching time - therefore the efficiency of a BHTP will increase with a decrease in the number of BIP's used.

In this paper, we consider an underlying combinatorial optimization problem, one that finds the BIPs and their associated {\em weights} such that for each beam, the ratio of accumulated weights for this beam over total accumulated weights from all beams is as close to 
$\hRb/(\sum_{b \in \mN} \hRb)$ 
as possible for a fair distribution of data rate, whilst minimizing the number of BIPs required. 
This optimization problem assumes each beam in a BIP achieves the same data transfer rate, which is a function of the defined satellite capacity and the number of simultaneously illuminated beams. When considering real-world integration of a BH system, the supplied data rate to a beam is a function of the assigned bandwidth and spectral efficiency - which will be determined by the modulation and coding scheme used based on the achieved carrier to noise ratio (CNR). Therefore the solution to the combinatorial optimization problem provides an objective metric for further resource planning, as the power and bandwidth required to achieve the desired data transfer rate when considering link noise and losses captured in an individual beams CNR can be determined.

\subsection{Literature review} 

The BHTP has been studied in the literature with different variations considered. Some divide the beams into clusters that allow only one beam to be illuminated in each cluster at a time, while others applied frequency re-use within each cluster for avoiding adjacent beam interference. In terms of heuristic methods, there are simple heuristic construction methods (with no real optimization), meta-heuristic approaches (with Genetic Algorithm (GA) the most commonly used), and Machine Learning (ML) methods (technically also heuristics). Even though there were Mathematical Programming (MP) formulations, as the calculation of supplied data rate requires the consideration of signal to interference and noise ratio (SINR), they are either not linear, or linear but exponentially large (e.g., a column generation-type formulation), and solved by heuristic approaches, one way or another. 


GA was implemented in, e.g., \cite{angeletti2006beam} and \cite{anzalchi2010beam}. The latter proposed to apply,  sequentially, GA, Variable Neighbour Search (VNS), and Iterated Local Search (ILS) in finding the optimal BHTP, power plan, and frequency plan. 
%
Iterative or incremental heuristic allocation has been proposed in, e.g., \cite{alberti2010system}. In \cite{alegre2012offered}, a sequential construction method is applied for time allocation and power allocation. 
In \cite{Shi2020}, a heuristic approach called the Cuckoo search is implemented. 
In \cite{lei2011multibeam}, when co-channel interference is considered, a heuristic iterative approach is implemented where in the first instance, carrier (frequency bandwidth) or time is allocated on a per-beam basis.  
Then, the power to be allocated to the selected carriers from the power constraint will be calculated. When co-channel interference is not considered, the optimization problems are formulated as convex programming problems, and solved using Lagrangean relaxation.  
In Mathematical Programming approaches,  \cite{Wang2021} developed a Mixed Integer Conic Programming model but proposed a greedy algorithm with polynomial time complexity. \cite{Wang2019} proposed a method for power and time-slot allocation, where all beams are divided into clusters whilst allowing four-colour frequency reuse within each cluster. The first optimization problem is to find the optimal number of clusters, then the second allocate power to clusters, and the third allocate time-slot to the beams. Again, Lagrangean relaxation is used.  \cite{Kyrgiazos2013} solves a binary integer programming problem to determine the BIP for each time-slot. If the BIP is different for each time-slot, then a substantial amount of time will be used in pre- and post-amble as opposed to used in data transfer.  
In ML approaches, \cite{Hu2019,hu2020dynamic,Zhang2019} proposed Deep Reinforcement Learning (DRL) approaches. An integrated Deep Learning (DL) and ILP approach is proposed in  \cite{Lei2020}, where DL is used to predict the most frequently used number of illuminated beams in each pattern, hence the ILP is only required to consider BIPs with the number of illuminated beams limited to the range identified by the DL process. 

Table \ref{T:equipos} presents a comparison of the key system parameters used to define problem instances in the reviewed literature. This highlights the difficulties in conducting direct comparisons between solution approaches as each paper used different input datasets and system assumptions. While most of the reviewed literature constrained the problem to a specific geographic location, it was noted that the reviewed papers did not present orbital parameters for the reference satellite and the impact of satellite beam elongation was not discussed. While most papers utilised a flat beam projection (assuming beam shapes remain unchanged from that at the sub-satellite point), this assumption was not justified.


\begin{table*}[!h]
\begin{centering}
\begin{tabular}{| l | c | c | c | c  |c|}
\hline
\textbf{Literature} & \multicolumn{4}{ c |}{\textbf{System parameters}} & \textbf{Method} \\ 
\cline{2-5}
&  &      & Timeslot /     & Coverage Region /  & \\
& $N_B$ & $N_{\max}$     &  BH Cycle Duration    & Projection & \\
\hline 
\cite{alberti2010system} 
& 70    & 18SP/35DP  & Flexible / Undefined  & Europe / Flat    
& ICH \\ \hline
 \cite{angeletti2006beam}
 & 100  & 35       &   Undefined / 12 slots            & Europe / Flat       
 & GA \\ \hline
 \cite{anzalchi2010beam} 
 & 70   &    Undefined       &           Undefined / 4 and 12 slots    & Europe / Flat     
 & GA - (V)NS - ILS \\ \hline
 \cite{Hu2019} 
 & 37   & 10          &     100ms / Undefined           & Undefined / Flat
 & DRL \\ \hline
 \cite{hu2020dynamic} 
 & 37   & 10           & 100ms / Undefined    
                                                             & Undefined / Spherical
 & DRL \\ \hline
 \cite{lei2020beam}  
 & 16   &    Unconstrained      & 1ms / 256 slots     
                                                         & Europe / Spherical
 & IP/LP-DL\\ \hline
 \cite{Wang2019a}  
 & 70   &   5       & Undefined / 100 slots     
                                                         & Europe / Flat
 & ICH\\ \hline
\end{tabular}
\caption{Abbreviations: 
GA, genetic algorithm; 
SP, single polarization; 
DP, dual polarization; 
ICH, iterative constructive heuristics; (V)NS, (variable) neighbourhood search; ILS, iterated local search; DL, deep learning; DRL, deep reinforcement learning; IP, integer programming; LP, linear programming. Recall that $N_B$ is the number of user beams, $N_{\max}$ is the maximum number of beams that can be illuminated simultaneously. 
} 
\label{T:equipos}
\end{centering}
\end{table*}

\subsubsection{Contributions and outline of paper}

Whilst existing state-of-the-art algorithms (for BH and other resource allocation problems, such as power allocation) use various forms of matching of supplied data rate to requested data rate as an objective function in the model formulation, we propose a very different approach. We minimize the number of patterns within a BHTP, reducing the time assigned to beam switching and terminal synchronization within a hopping cycle, increasing the effective illumination time per cycle. We then calculate the accumulated time allocation to each beam in the BHTP such that total illuminated time of each beam conforms to the requested data rate distribution to the beams. Such an approach maximizes the efficiency of a BHTP while ensuring a fair supply distribution, in that each beam will have more or less the same demand-satisfaction percentage (i.e., the same ratio of unmet and supplied data). With this approach, we have a well-defined combinatorial optimization problem, which we describe in Section \ref{Sec:IPs} where we present a Binary Integer Linear Programming (BILP) model followed by a Mixed Integer Linear Programming (MILP) model. In terms of other resources, we consider a uniform power allocation to each beam.

  Our approach does not require handling the CNR within the model, as satellite throughput is assumed to be equal to the requested demand rate, allowing for a generalized optimization approach, not specific to a satellite or geographic region.  

Both mathematical programming models can only solve problems of modest scales as they both have a very poor Linear Programming relaxation (LPR) lower bound (LB)--we prove that the LPR LB is 1 for both models. For this reason, we propose a rapid complexity log $n$ algorithm -- the Decomposition by Powers of Two (DP2) algorithm, in Section \ref{Sec:DP2}. DP2 can solve a 1085-beam problem in just 19.01 seconds. The capacity error is less than 10\% for up to 100 beams, and under 19\% for up to 200 beams.   

Due to the lack of common data sets, satellite models and objective functions, a direct comparison between studies is difficult. In order to create a benchmark problem set, we created a satellite agnostic test-bed with simulated beam footprints, user demands and computational results as displayed in Section \ref{Sec:results}.

\section{Mathematical models} 
\label{Sec:IPs} 

In this section, we present a BILP for fixed-length SF BH systems, and a MILP for both fixed-length and flexible-length SF BH systems. We then present some mathematical properties of these mathematical programming formulations. 

The two mathematical programming models will achieve the following goal. Suppose that we have 5 beams, and the requested demands are given by $\hat{R}_1 = 109, \hat{R}_2 = 120, \hat{R}_3 = 91, \hat{R}_4 = 87, \hat{R}_5 = 135$. Instead of minimizing the differences between the offered and requested data rate, our solutions will ensure that the total illuminated times of the five beams will conform to the following ratio: $\frac{109}{542}, \frac{120}{542}, \frac{91}{542}, \frac{87}{542}, \frac{135}{542}$, hence ensuring the fairness to all 5 beams. We minimizing the total number of patterns required, hence reducing the times needed for switching and terminal synchronization. In other words, the models will maximize the effective illumination time per cycle and allocate the same \% demand satisfaction for all beams.

\subsection{A BILP formulation for fixed-length SF} 

A BIP can be represented mathematically by ones and zeros, with a ``1'' indicating a beam is illuminated and a ``0'' indicating otherwise. 
First, we consider the case when the SF duration is fixed and measured as number of time slots instead of the actual time. 

Let: 
\begin{itemize}
    \item $\Omega$ be the number of distinct demands in $\{\hRb \ | \ b \in \mN\}$; 
    \item $\mH = \big\{1,\ldots,\min \{W, \Omega\}\big\}$ be the set of weights; 
    \item $\zph \in \{0,1\}$ with $\zph=1$ if BH Pattern $p$ has a beam dwell time of $h$, and  0 otherwise;
    \item $\yphb \in \{0,1\}$  if Beam $b$ in BH Pattern $p$ and has a weight of $h$, and 0 otherwise; and 
    \item $D_b$ is the per BH cycle demand of Beam $b$ (the reason we are using a new notation here is that the demand is used as the weight in our framework, not to be confused with the actual requested or offered data rate); 
    \item $\beta(b) \subset \mN$ be the set of beams that are neighbouring (adjacent) Beam $b$; 
    and 
    \item $M = \max\{ D_b \ | \ b \in \mN\}$. 
\end{itemize} 
Recall that $\Nmax$ is the maximum number of beams that can be simultaneously illuminated. Notice that the set $\mP = \{1,\ldots,|\mP|\}$ is enumerated as index set of the patterns, for $|\mP|$ an upper bound on the number of BIPs required,  can be bounded above by the rapid heuristic algorithm we present in Section \ref{Sec:DP2}. If the set $\mP$ contains more BIPs than needed, there will be patterns $p$ such that $z^{p,h} = 0$ for all $h$. 
\begin{model}
\label{Model1}
\begin{alignat}{3}
w &= \min  \sum_{p \in \mP} \sum_{h \in \mH}& & \zph \label{M1:ObjFn} \\
 \mbox{s.t. } 
 &   \yphb & \ \leq \ &  \zph, &  \forall p \in \mP, h \in \mH, b \in \mN\label{M1:Logic} \\
 %
 %
 &   \sum_{h \in \mH}  \zph & \ \leq \ & 1, & \ \forall p \in \mP,   \label{M1:oneDwellTime} \\
 %
 %
  &  \small{\sum_{p \in \mP}  \frac{\displaystyle\sum_{h \in \mH} h \yphb}{ \displaystyle\sum_{b'' \in \mN} \sum_{h \in \mH} \yphbd } } & \ = \ & D_b, & \ \forall b \in \mN \label{M1:requestedDataRate} \\ 
 %
%
 &  \sum_{b \in \mN} \sum_{h \in \mH}  \yphb 
 & \ \leq \ & N_{\max}, & \ \forall p \in \mP \label{M1:cardinalityUB}  \\
 & y^{p,h}_b +  y^{p,h}_{b'} & \ \leq \ & 1, & \ \forall p \in \mP, h \in \mH, \notag \\
 & & & & b' \in \beta(b) \label{M1:interference}
\end{alignat} 
\end{model}

The objective function (\ref{M1:ObjFn}) minimizes the total number of number of BIPs required. The logic constraint, (\ref{M1:Logic}), ensures that Beam $b$ can only be illuminated in Pattern $p$ with weight $h$ if such a pattern-weight pair exists in the solution.  
Each BHP can only be associated with one BDT in  (\ref{M1:oneDwellTime}). 
In (\ref{M1:requestedDataRate}), we require that the sum of the weights overall all BIPs where Beam $b$ is illuminated equal to the per BH cycle demand $D_b$. (Notice that if BIP $p$ has $X$ beams illuminated, the power will be evenly distributed over each beam, hence divided by $X$). 
Constraint  (\ref{M1:cardinalityUB}) calculates the cardinality of each BHP and ensures that no more than $\Nmax$ beams are illuminated simultaneously. 
(\ref{M1:interference}) ensures no adjacent beams are illuminated simultaneously. 

Post-optimization, we calculate the total weight in an optimal solution, given as: $H^* =  \sum_{p \in \mP} \sum_{h \in \mH}  h(\zph)^*$, for $(\zph)^*$ the optimal values of the $z$ variables. 
If $T_H/ H^* \geq m_d$, for $m_d$  the minimum granularity permitted in existing HTS system, we set $W=H^*$. Otherwise, we adjust the weights $h$ in each BIP $p$ by setting the new weight $h'$ to be $h' = [h (W/H^*)]$.

As Constraint (\ref{M1:requestedDataRate}) is non-linear, our strategy is to solve a simplified but linear model and modify the power required post-optimisation. 
\begin{model}
\label{Model2}
\begin{alignat}{3}
    \min \ (\ref{M1:ObjFn}) \mbox{ s.t. } (\ref{M1:Logic}),  (\ref{M1:oneDwellTime}), 
    (\ref{M1:interference}) \mbox{ and } \notag \\
  \sum_{p \in \mP} \sum_{h \in \mH} h \yphb  & \ = \ & D_b, & \ \forall b \in \mN \label{M1:requestedDataRate2}
\end{alignat}
\end{model} 
Let $z^* , y^*$ be the optimal solution to Formulation \ref{Model2}, and $L_p = \sum_{h \in \mH}  \sum_{b \in \mN} (y^{p,h}_b)^*$. 
Method 1. We supply $L_p$ times more power to the BIP $p$. Method 2. We set $H^{**}$ to be $H^{**} = \sum_{p \in \mP} \sum_{h \in \mH} (z^{p,h})^* L_p$. Again, if $T_H / H^{**} < m_d$, we set the new weight $h'$ for BIP $p$ to be $h' = [h L_p \times \frac{W}{H'}]$. 
\begin{proposition}
The optimal value of the linear programming relaxation of Formulation \ref{Model2} is 1. 
\end{proposition}

\begin{proof}

We prove the proposition by first showing that the optimal objective value of the (LPR) Formulation \ref{Model2} is at least 1, then showing that there exists a solution with objective 1, hence deducing the optimal objective value is exactly 1.  
Let $b^* = \mbox{arg} \max \{D_b \ | \ b \in \mN\}$. 
As (\ref{M1:requestedDataRate2}) must be true for all $b \in \mN$, it is true for $b = b^*$. 
This gives us $\sum_p \sum_h h\yphbs = M$, which implies that $ \sum_p \sum_h (h/M) \yphbs = 1 $, and therefore $\sum_p \sum_h \yphbs  \geq 1$. As (\ref{M1:Logic}) must be true for all $b \in \mN$, it is true for $b = b^*$, and thus $w^* \geq 1$.  
Now we 
show that there exists a feasible solution $\hzph, \hyphb$ such that $\sum_p \sum_h \hzph =1$. 

Consider the solution defined by: 
$\hzph, \hyphb$ such that: 
\begin{enumerate}
    \item $\hzph = 0$ for all $p$ and $h =1,\ldots, M-1$;
    \item $\hyphb = 0$ for all $p$, $b$ and all $h =1,\ldots, M-1$; 
    \item $\hypMb$ chosen such that $\sum_p \hypMb = D_b / M$ for all $b$,  
    and $\hypMb \leq 0.5$ for all $p$, $b$ (hence satisfying (\ref{M1:requestedDataRate2}), 
    and (\ref{M1:interference})); 
    \item $\hzpM = \hat{y}^{p,M}_{b^*}$ for all $p$ (hence satisfying (\ref{M1:Logic})); 
\end{enumerate}
We have that $\sum_p \hzpM = \sum_p \hat{y}^{p,M}_{b^*} = M / M = 1$ and (\ref{M1:oneDwellTime}) is satisfied as $\hzpM$ is a linear relaxation of $\zph \in \{0,1\}$. Hence the proposition is proved.  \end{proof}

Once the weights and BDTs are determined, the actual supplied data $\bRb$ can be calculated by applying suitable power to the each beam in the BIPs, taking into consideration the CNR.

\subsection{A MILP formulation for flexible-length SF} 

We now present a reformulation, a MILP that solves the same combinatorial optimization problem but can also handle systems that allow flexible-length SFs as it allocates the BDTs for each BIP directly. 
We first define some new notation. 
Let: 
\begin{itemize}
    \item $z^p \in \{0,1\}$ with $z^p=1$ if Pattern $p$ is used in the BH cycle and 0 otherwise; 
    \item $y^p_b \in \{0,1\}$ with $y^p_b=1$ if Beam $b$ is illuminated in Pattern $p$ and 0 otherwise; 
    \item  $T^p \in \mathbb{R}_+$ the BDT of Pattern $p$;  
    \item $x^p_b \in \mathbb{R}_+$ if Beam $b$ is illuminated in Pattern $p$, (the illumination duration should be the same as Pattern $p$'s BDT if the pattern is used in the BH cycle).
\end{itemize} 

\begin{model}
\label{Model3} 
\begin{alignat}{3}
 & w^*  & \ = \ &  \min \sum_{p \in \mP}  z^p \label{M2:ObjFn} \\
 \mbox{s.t. } 
 &   y^p_b & \ \leq \ &  z^p, &  \forall p \in \mP,b \in \mN\label{M2:Logic1} \\
 &   x^p_b & \ \leq \ &   D_b y^p_b, &  \forall b \in \mN\label{M2:Logicw} \\
  &  \sum_{p \in \mP} x^p_b  & \ = \ & D_b, & \ \forall b \in \mN \label{M2:requestedDataRate} \\
 &   x^p_b & \ \leq \ & T^p + M (1-y^p_b), & \ \forall p \in \mP, b \in \mN  \label{M2:sameDwellTime1} \\
 &  T^p & \ -\ & M (1-y^p_b) \leq  x^p_b, & \ \forall p \in \mP, b \in \mN  \label{M2:sameDwellTime2} \\
  & \sum_{b \in \mN} y^p_b  & \ \leq \ & \Nmax,  & \  \forall b \in \mN  \label{M2:cardinalityUB} \\
  & y^p_b +  y^p_{b'} & \ \leq \ & 1, & \ \forall p \in \mP,  b' \in \beta(b) \label{M2:interference} 
 %
\end{alignat} 
\end{model}
The objective function (\ref{M2:ObjFn}) counts the number of BIPs used. The logic constraint (\ref{M2:Logic1}) enforces that a beam is only illuminated in a pattern if the pattern is used, and  (\ref{M2:Logicw}) is another logic constraint to ensure that a beam will only be given a weight if it is illuminated in Pattern $p$. 
The sum of all weights for each beam is equal to the per BH cycle demand and is guaranteed by the use of Constraint  (\ref{M2:requestedDataRate}).
Constraints (\ref{M2:sameDwellTime1}) and (\ref{M2:sameDwellTime2}) ensures that the weight must be the same for all beams in the same pattern.  %
Cardinality restriction and neighbouring beam interference constraints are given in  (\ref{M2:cardinalityUB}) and (\ref{M2:interference}) respectively. 
Again, (\ref{M2:cardinalityUB}) is not always needed in real systems. 

The benefits of Formulation \ref{Model3} include: 1)  even though the $y^p_b$ and $T^p$ variables are continuous, the optimal solution is naturally integral if $D_b$ are integers for all $b\in \mN$; and 2) the $\zph$ and $\yphb$ variables in Formulations 1 and 2 requires that the set of possible weights enumerated in advance, hence intractable for large $D_b$. Formulation \ref{Model3} solves much faster than Formulation \ref{Model2} however the LPR LB is also poor.  We do not consider (\ref{M2:cardinalityUB}) is the following proposition as it is not always used in real systems. 
\begin{proposition}
The optimal value of the LP relaxation of Formulation \ref{Model3}, $w^*$, is 1. 
\end{proposition}
\begin{proof}
Substituting (\ref{M2:requestedDataRate}) into (\ref{M2:Logicw}), we have that $\sum_{p \in \mP} x^p_b = D_b \leq D_b \sum_{p \in \mP} y^p_b$, which implies that $\sum_{p \in \mP} y^p_b \geq 1$ for all $b \in \mN$ and thus $\sum_{p \in \mP} z^p \geq 1$. We now show that there exists a feasible fractional solution with an objective value of 1. Let $\mNgz = \{b \in \mN  :  D_b >0\} \subseteq \mN $. 
Consider the solution given by $\ydpb, \zdp, \Tdp, \xdpb$ such that:  \begin{enumerate}
    \item 
    $\sum_{p \in \mP} \zdp =1$, (which implies that by (\ref{M2:Logic1}), (\ref{M2:Logicw}), and (\ref{M2:requestedDataRate}), $\sum_{p \in \mP} \ydpb = 1 $ for all $b \in \mN$; 
     \item 
     $\xdpb = D_b \ydpb$, for all $p \in \mP$ and $b \in \mN$ (thus satisfying (\ref{M2:Logicw}), and (\ref{M2:requestedDataRate}));
    \item 
     $\ydpb = \hat{y}^p_{b'}  = \zdp$  for all $p \in \mP$ and all $b, b' \in \mNgz, \ b\neq b'$,  $\ydpb = 0$ otherwise, and $\zdp \leq 0.5$ for all $p \in \mP$ (thus satisfying (\ref{M2:interference})); 
     %
     \item 
     $T^p = 0$, for all $p \in \mP$, and by (\ref{M2:Logicw}),  satisfies (\ref{M2:sameDwellTime1}) and (\ref{M2:sameDwellTime2}) as $y^p_b \leq 0.5$ for all $p \in \mP$ and $b \in \mN$, and that $M \geq D_b$ for all $b \in \mN$. 
\end{enumerate} 

Hence, the proposition is proved. 
\end{proof}
Notice that all patterns in an optimal solution $p \in \mP$ with $z^p >0$ are distinct. If there exists two patterns $p_1, p_2$ with the same set of beams illuminated with BDTs $T^{p_1}$ and $T^{p_2}$ respectively, they can be combined into a single the patterns with a BDT of $T^{p_1} + T^{p_2}$.

In implementation, we added the following symmetry elimination constraints - ordering BIPs by their weights to remove symmetry by permutation. 
\begin{alignat}{3}
 T^p & \ \leq \ &  T^{p+1}, &  \quad \forall p \in \mP \setminus \{|\mP|\} \label{M2:symmetryElimination} 
\end{alignat}

\section{A rapid heuristic algorithm - Decomposition by powers of two} 
\label{Sec:DP2}

We now describe our decomposition by powers of two (DP2) heuristic. 
Let $P^j$ be the binary vector that represents the $j$th BIP and $h_j$ be the associated weight.  
 
\begin{algorithm}
\caption{Decomposition by powers of 2}\label{alg:greedy}
\begin{algorithmic}[1]
\State Input $\hat{R}$ 
\State Set $\hat{R}_{\max} = \max \{\hRb \ | \ b \in \mN\}$ 
\State Set $K = \lfloor\log_2 \hat{R}_{\max} \rfloor$
\While{$K \geq 0$}
\For{$b \in \mN$} 
\State Set $h_K = 2^K$
\State  Set $P^K_b = 1$ if $\hRb \geq 2^K$ and $P^K_b = 0$ otherwise 
\State set  $\hRb \leftarrow \hRb - h_K P^K_b$
\EndFor 
\State Set $K \leftarrow K-1$
\EndWhile
\State Output  $P^0, \ldots, P^K$, $h^0,\ldots,h^K$
\end{algorithmic}
\end{algorithm}

The DP2 method provides an upper bound to Formulations \ref{Model2} and \ref{Model3} as the solutions it produces are feasible to the BHOP but has an additional constraint that all weights must be powers of two. 

Notice that each of the Power of 2 Matrix (P2M) $H^j = h_j P^j$, for all $j = 1,\ldots,K$ may be required to be further decomposed if the cardinality and/or neighbouring beam interference constraints are enforced. To eliminate neighbouring beam interference, at most three BIPs is needed, so $H^j$ can be further decomposed into a maximum of three BIPs $H^{j,1}, H^{j,2}, H^{j,3}$, for all $j = 1,\ldots,K$. 
If there exists a limit in the maximum number of illuminated beams $\Nmax$, e.g., in $H^{j,2}$ if $|\{ h^{j,2}_b > 0 \ | \ \forall b \in \mB  \}| > \Nmax$, then at most 
\[
\left\lceil 
\dfrac{|\{ h^{j,2}_b > 0 \ | \ \forall b \in \mB  \}|}{\Nmax}
\right\rceil
\]
BIPs is required for $H^{j,2}$. 

At the end of all procedures, one can merge multiple appearances of the same pattern into one and use the sum of the individual illumination times. Merging them will increase effective illumination time by saving time spent in switching and terminal synchronization.

\section{Numerical experiments} 
\label{Sec:results}

\subsection{Performance of the integer programs} 

We tested Formulation \ref{Model2} (F2), Formulation \ref{Model3} (F3), and Formulation \ref{Model3} with symmetry elimination (\ref{M2:symmetryElimination}) (F3S) on two small instances under 3 settings using IBM CPLEX V20.1 on a Mac Pro with 3 GHz 8-Core Intel Xeon E5 processor and 32 GB memory. 
\\ 
Instance 1 (INS1): 10 beams, max $D_b=39$, min  $D_b=2$.  \\
Instance 2 (INS2): 
15 beams, max $D_b=156$, min $D_b=20$. \\
Setting 1: DP not used as an UB, no cardinality constraints, no adjacent beam interference elimination constraints. \\
Setting 2: DP used as an UB, no cardinality constraints, no adjacent beam interference elimination constraints. \\
Setting 3: DP used as an UB, with cardinality and adjacent beam interference elimination constraints. \\
Solution times are in seconds. A limit of 300 seconds is given for all instances and if the problem is not solved to optimality, the optimality gap is presented. 

\begin{table}[!h]
\label{T:mathmodels}
\begin{center}
\begin{tabular}{llll}
& F2 & F3 & F3S  \\ \hline
Setting 1 - INS1 (opt)  & 5 & 5  & 5    \\
Setting 1 - INS1 (time) & 3.83  & 0.24 & 0.42    \\ \hline
Setting 1 - INS2 (opt)  & 40.76\%  & 42.58\% & 50.24\%   \\
Setting 1 - INS2 (time) & 300 & 300  & 300   \\ \hline
Setting 2 - INS1 (opt)  & 5  & 5 & 5   \\
Setting 2 - INS1 (time) & 1.01 & 0.07 & 0.06      \\ \hline
Setting 2 - INS2 (opt)  & 26.88\%  & 6 & 6  \\
Setting 2 - INS2 (time) & 300 &  65.32 & 47.03 \\ \hline
Setting 3 - INS1 (opt)  & 8 & 8  & 8   \\
Setting 3 - INS1 (time) & 53.84  & 0.68  & 0.66    \\ \hline
Setting 3 - INS2 (opt)  & 55.39\% & 20\% & 9  \\
Setting 3 - INS2 (time) & 300  & 300  & 75.18 
\end{tabular}
\caption{Comparisons of the mathematical programming models. Computation times are in seconds, and the optimality gaps are reported as percentages for instances where a proven optimality is achieved. A computation time limit of 300 seconds is applied to all problem instances. The data instances are presented in Appendix A.}
\end{center}
\end{table}

Clearly Formulation \ref{Model3} with symmetry elimination constraints performed the best. We tested a 49-beam problem with max demand 50, without neighbouring beam interference consideration, the problem is solved under Settings 1 \& 2 within seconds. The optimal objective value is the same as the one returned by the DP2 algorithm. 
Once the cardinality and adjacent beam interference elimination constraints are added, CPLEX did not return a feasible solution after running for 10 minutes. 

\subsection{Performance of the DP2 algorithm}
To assess performance of the DP2 method in medium to large problem instances, a geographic and satellite agnostic test-bed was developed. The satellite is assumed to service demand points located within a 100x100 unit grid centred at the origin. The beam radiation pattern is generated using overlapping circular footprints of uniform radius representing the half-power bandwidth of a parabolic reflector and inter-beam interference is assumed to be negligible beyond adjacent beams (as shown in Figure \ref{Dia:solution}). User locations are aggregated to the centres of each beam servicing the user, with demands adjusted such that the ratio between the carrier to noise ratio ($C/N_0$) and the user demand ($R_b$) remains constant when considering the change in effective isotropic radiated power (EIRP) and antenna gain to noise temperature ratio ($G/T$) at the aggregated location. 

All trials were conducted with a fixed SF duration of 1.5ms and a complete cycle duration of 256 SFs. Three input factors were identified for numerical experimentation, with a two-level input coding as follows:
\begin{table}[!h]
\label{T:dataset}
\begin{center}
\begin{tabular}{cccc}
\multicolumn{1}{l}{\textbf{Trial}} & \multicolumn{1}{l}{\textbf{\# Users}} & \multicolumn{1}{l}{\textbf{Demand Variation}} & \multicolumn{1}{l}{\textbf{Distribution}} \\ \hline 
\textbf{1}                         & 800                                   & 10-15 mbps                                    & Continuous                               \\
\textbf{2}                         & 200                                   & 1-35 mbps                                     & Continuous                               \\
\textbf{3}                         & 800                                   & 10-15 mbps                                    & Discrete                                  \\
\textbf{4}                         & 200                                   & 1-35 mbps                                     & Discrete                                  \\
\textbf{5}                         & 200                                   & 10-15 mbps                                    & Discrete                                  \\
\textbf{6}                         & 800                                   & 1-35 mbps                                     & Discrete                                  \\
\textbf{7}                         & 200                                   & 10-15 mbps                                    & Continuous                               \\
\textbf{8}                         & 800                                   & 1-35 mbps                                     & Continuous                              
\end{tabular}
\caption{Characteristics of the benchmark problem set, the satellite agnostic test-bed with a simulated beam footprints and user demands. } 
\end{center}
\end{table}
Each trial was run with 16, 49 and 132 beams, user points were generated randomly, either continuously across the grid or in discrete clusters and demand values were generated randomly between upper and lower bounds.

\textbf{Capacity Error.} Relative capacity error is the sum of absolute difference of supplied and requested data rates across all beams, as a proportion of the total demand. BIPs generated by DP2 significantly reduced capacity error across all trials, with an average error reduction of 94.3\% (Figure \ref{Dia:USC_small}). For an even distribution, capacity error was significantly impacted by user distribution while the DP2 method displayed consistent results across all trials.

\begin{figure}[h!]
\begin{center}
\centerline{\includegraphics[width=0.45\textwidth]{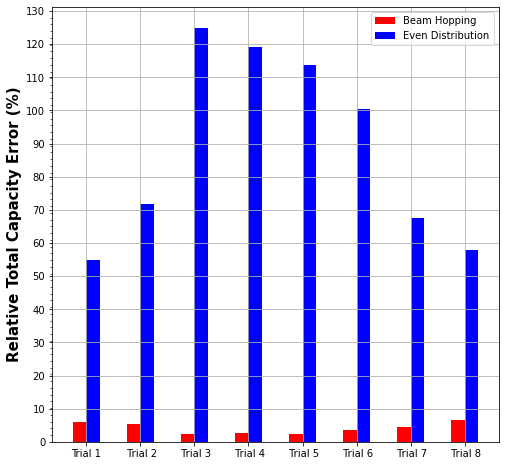}}
\caption{Capacity error for BH and even distribution systems}
\label{Dia:USC_small}
\end{center}
\vskip -0.2in
\end{figure}

\textbf{Process Time and Pattern Numbers.} The ratio of total BIPs to beam numbers ($B_{ratio}$) significantly decreases with increased beams (Figure \ref{Dia:P_time}). This is due to the impact of separating patterns with adjacent beams being more profound for larger beamwidths. For medium to large problem instances, the trivial upper bound of one pattern per beam ($B_{ratio}=1$) is significantly improved using DP2, with 49 and 132 beam systems having an average $B_{ratio}$ of 0.28 and 0.09 respectiveleven.  

\begin{figure}[h!]
\begin{center}
\centerline{\includegraphics[width=0.45\textwidth]{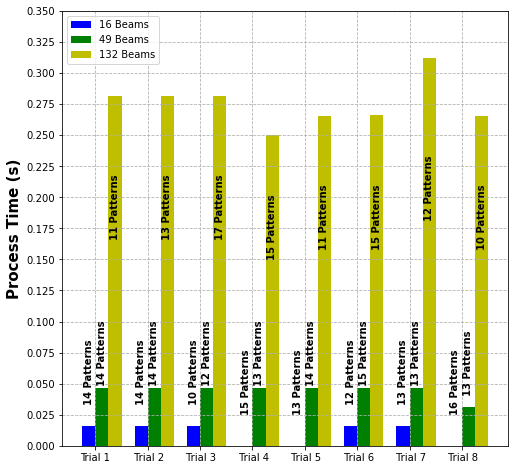}}
\caption{Process time and required number of decomposed beam illumination patterns}
\label{Dia:P_time}
\end{center}
\vskip -0.2in
\end{figure} 

\textbf{Extended Process Time.} Process time to generate BIPs using DP2 remained tractable up to the tested limit of 1085 beams, with a solution obtained in 19.01s (Figure \ref{Dia:time}). There is an increase in capacity error with the number of beams. As the DP2 method decomposes demands to the nearest integer, accumulated rounding error will increase with beam numbers due to individual beams having a smaller portion of overall demand (increased sensitivity to rounding) and more rounding instances (potentially one instance per beam).

\begin{figure}[h!]
\begin{center}
\centerline{\includegraphics[width=0.45\textwidth]{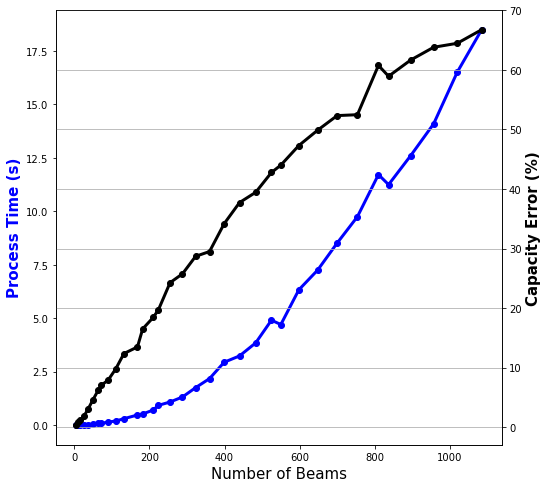}}
\caption{Process time and capacity error for increasing beam numbers}
\label{Dia:time}
\end{center}
\vskip -0.2in
\end{figure} 

In terms of the optimality gap of the DP2 algorithm, we refer to the two problem instances generated for the integer programming models. For Instance 1, as the largest power of 2 that is no greater than max $D_b=39$ is $2^5$, together with $2^0$, we have six patterns, 1 more than the optimal solution for Settings 1 and 2 for Instance 1 returned by the exact mathematical formulations. For Instance 2, the largest power of 2 no greater than max $D_b=156$ is $2^7$, so there will be eight patterns for Settings 1 and 2, same as the optimal solution returned by the exact mathematical formulations.

\section{Conclusions and future research direction} 

In this paper we presented a rapid complexity log $n$ heuristic algorithm, DP2, for solving BHTP problems. We also presented two integer programming models for exact solutions and proved the lower bounds of the linear programming relaxation. The BILP and MILP models, as they are presented cannot handle large-scale problem instances. A future research direction is to develop custom-made algorithms for solving larger problem instances within reasonable computation time. DP2, on the other hand, solved instances with over 1000 beams in 19.109 seconds. The benefits of our DP2 algorithm are that it is computationally efficient and simple to implement with the ability to adapt to varying input scenarios while maintaining fairness (beam demand matching). Noting the impact of cumulative rounding errors, the DP2 method is well suited for medium sized instances (50-200 beams) or for large instances ($>$200 beams) when used as an upper bound for complex optimization algorithms. As a future research direction, we plan to move from a generic test-bed to system specific testing, assessing potential integration issues.

\bibliographystyle{named}
\bibliography{ijcai22}

\begin{thebibliography}{}

\bibitem[\protect\citeauthoryear{Alberti \bgroup \em et al.\egroup }{2010}]{alberti2010system}
X~Alberti, J~M Cebrian, A~{Del Bianco}, Z~Katona, J~Lei, M~A Vazquez-Castro, A~Zanus, L~Gilbert, and N~Alagha.
\newblock {System capacity optimization in time and frequency for multibeam multi-media satellite systems}.
\newblock In {\em 2010 5th Advanced Satellite Multimedia Systems Conference and the 11th Signal Processing for Space Communications Workshop}, pages 226--233. IEEE, 2010.

\bibitem[\protect\citeauthoryear{Alegre-Godoy \bgroup \em et al.\egroup }{2012}]{alegre2012offered}
Ricard Alegre-Godoy, Nader Alagha, and Maria~Angeles V{\'{a}}zquez-Castro.
\newblock {Offered capacity optimization mechanisms for multi-beam satellite systems}.
\newblock In {\em 2012 IEEE International Conference on Communications (ICC)}, pages 3180--3184. IEEE, 2012.

\bibitem[\protect\citeauthoryear{Angeletti \bgroup \em et al.\egroup }{2006}]{angeletti2006beam}
Piero Angeletti, David {Fernandez Prim}, and Rita Rinaldo.
\newblock {Beam hopping in multi-beam broadband satellite systems: System performance and payload architecture analysis}.
\newblock In {\em 24th AIAA International Communications Satellite Systems Conference}, page 5376, 2006.

\bibitem[\protect\citeauthoryear{Anzalchi \bgroup \em et al.\egroup }{2010}]{anzalchi2010beam}
J~Anzalchi, A~Couchman, P~Gabellini, G~Gallinaro, L~D'agristina, N~Alagha, and P~Angeletti.
\newblock {Beam hopping in multi-beam broadband satellite systems: System simulation and performance comparison with non-hopped systems}.
\newblock In {\em 2010 5th Advanced Satellite Multimedia Systems Conference and the 11th Signal Processing for Space Communications Workshop}, pages 248--255. IEEE, 2010.

\bibitem[\protect\citeauthoryear{Hu \bgroup \em et al.\egroup }{2019}]{Hu2019}
Xin Hu, Shuaijun Liu, Yipeng Wang, Lexi Xu, Yuchen Zhang, Cheng Wang, and Weidong Wang.
\newblock {Deep reinforcement learning-based beam Hopping algorithm in multibeam satellite systems}.
\newblock {\em IET Communications}, 13(16), 2019.

\bibitem[\protect\citeauthoryear{Hu \bgroup \em et al.\egroup }{2020}]{hu2020dynamic}
Xin Hu, Yuchen Zhang, Xianglai Liao, Zhijun Liu, Weidong Wang, and Fadhel~M Ghannouchi.
\newblock {Dynamic beam hopping method based on multi-objective deep reinforcement learning for next generation satellite broadband systems}.
\newblock {\em IEEE Transactions on Broadcasting}, 66(3):630--646, 2020.

\bibitem[\protect\citeauthoryear{Kyrgiazos \bgroup \em et al.\egroup }{2013}]{Kyrgiazos2013}
A.~Kyrgiazos, B.~Evans, and P.~Thompson.
\newblock {Irregular beam sizes and non-uniform bandwidth allocation in HTS}.
\newblock In {\em 31st AIAA International Communications Satellite Systems Conference}, 2013.

\bibitem[\protect\citeauthoryear{Lei and Vazquez-Castro}{2011}]{lei2011multibeam}
Jiang Lei and Maria~Angeles Vazquez-Castro.
\newblock {Multibeam satellite frequency/time duality study and capacity optimization}.
\newblock {\em Journal of Communications and Networks}, 13(5):472--480, 2011.

\bibitem[\protect\citeauthoryear{Lei \bgroup \em et al.\egroup }{2020a}]{lei2020beam}
Lei Lei, Eva Lagunas, Yaxiong Yuan, Mirza~Golam Kibria, Symeon Chatzinotas, and Bj{\"{o}}rn Ottersten.
\newblock {Beam illumination pattern design in satellite networks: Learning and optimization for efficient beam hopping}.
\newblock {\em IEEE Access}, 8:136655--136667, 2020.

\bibitem[\protect\citeauthoryear{Lei \bgroup \em et al.\egroup }{2020b}]{Lei2020}
Lei Lei, Eva Lagunas, Yaxiong Yuan, Mirza~Golam Kibria, Symeon Chatzinotas, and Bjorn Ottersten.
\newblock {Deep Learning for Beam Hopping in Multibeam Satellite Systems}.
\newblock In {\em IEEE Vehicular Technology Conference}, volume 2020-May, 2020.

\bibitem[\protect\citeauthoryear{Shi \bgroup \em et al.\egroup }{2020}]{Shi2020}
Dingyuan Shi, Feng Liu, and Tao Zhang.
\newblock {Resource Allocation in Beam Hopping Communication Satellite System}.
\newblock In {\em 2020 International Wireless Communications and Mobile Computing, IWCMC 2020}, 2020.

\bibitem[\protect\citeauthoryear{Wang \bgroup \em et al.\egroup }{2019a}]{Wang2019}
Lin Wang, Chen Zhang, Dexin Qu, and Gengxin Zhang.
\newblock {Resource allocation for beam-hopping user downlinks in multi-beam satellite system}.
\newblock In {\em 2019 15th International Wireless Communications and Mobile Computing Conference, IWCMC 2019}, 2019.

\bibitem[\protect\citeauthoryear{Wang \bgroup \em et al.\egroup }{2019b}]{Wang2019a}
Yaxin Wang, Dongming Bian, Jing Hu, Jingyu Tang, and Chuang Wang.
\newblock {A Flexible Resource Allocation Algorithm in Full Bandwidth Beam Hopping Satellite Systems}.
\newblock In {\em Proceedings of 2019 IEEE 3rd Advanced Information Management, Communicates, Electronic and Automation Control Conference, IMCEC 2019}, 2019.

\bibitem[\protect\citeauthoryear{Wang \bgroup \em et al.\egroup }{2021}]{Wang2021}
Anyue Wang, Lei Lei, Eva Lagunas, Symeon Chatzinotas, Ana Isabel~Pérez Neira, and Bj{\"{o}}rn Ottersten.
\newblock Joint beam-hopping scheduling and power allocation in noma-assisted satellite systems.
\newblock In {\em 2021 IEEE Wireless Communications and Networking Conference (WCNC)}, pages 1--6, 2021.

\bibitem[\protect\citeauthoryear{Zhang \bgroup \em et al.\egroup }{2019}]{Zhang2019}
Yuchen Zhang, Xin Hu, Rong Chen, Zhili Zhang, Liquan Wang, and Weidong Wang.
\newblock {Dynamic beam hopping for DVB-S2X satellite: A multi-objective deep reinforcement learning approach}.
\newblock In {\em Proceedings - 2019 IEEE International Conferences on Ubiquitous Computing and Communications and Data Science and Computational Intelligence and Smart Computing, Networking and Services, IUCC/DSCI/SmartCNS 2019}, 2019.

\end{thebibliography}

\appendix

\section{Test Instance 1 for the mathematical programming model} 

 nbBeams $= 10$; 

  \noindent 
 Demands $= [
 20, 56, 30, 28, 15, 
 2, 39, 10, 38, 29
 ];$ 

 \noindent 
 beamNeighbours $= [
 \{2,4,5\}, 
 \{3,5,6\}, 
 \{6,7\}, 
 \{5,8\}$, 
 $
 \{6,8,9\},
 \{7,9,10\},
 \{10\},
 \{9\}, 
 \{10\} 
 ];$
 
\section{Test Instance 2 for the mathematical programming model}

  nbBeams $= 15$; 
  
  \noindent 
 Demands $= [
 20, 94, 89, 109, 63, 
 156, 70, 93, 86, 50
, 132$, 
$101, 28, 31, 39 
 ]; $
 
  \noindent 
 beamNeighbours $= [
\{2,6,7\}, 
\{3,7,8\}, 
\{4,8,9\}, 
\{5,9,10\}$,  
$\{10\}, 
\{7,11\}, 
\{8,11,12\}, 
\{9,12,13\}, 
\{10,13,14\}, 
\{14,15\}$, 
$\{12\}, 
\{13\}, 
\{14\},
\{15\}
 ];$
\end{document}